\documentclass{amsart}
\usepackage{tikz-cd}
\usepackage{latexsym,amssymb}
\usepackage[english]{babel}
\usepackage{bm}
\usepackage{amsfonts, amsthm}
\usepackage{amsmath}
\usepackage{mathrsfs}
\usepackage{multirow}
\usepackage{bm}
\usepackage{pgf,tikz}
\usepackage{hyperref}
\usepackage{algpseudocode}
\usepackage{algorithm}

\algnewcommand{\IIf}[1]{\State\algorithmicif\ #1\ \algorithmicthen}
\algnewcommand{\EndIIf}{\unskip\ \algorithmicend\ \algorithmicif}

\newtheorem{thm}{Theorem}[section]
\newtheorem{cor}[thm]{Corollary}
\newtheorem{ex}[thm]{Example}

\newtheorem{rem}{Remark}[section]

\begin{document}

\title{The maximum cardinality of trifferent codes with lengths $5$ and $6$}

\author[S. Della Fiore]{Stefano Della Fiore}
\address{DII, Universit\`a degli Studi di Brescia, Via
Branze 38, I-25123 Brescia, Italy}
\email{s.dellafiore001@unibs.it}

\author[A. Gnutti]{Alessandro Gnutti}
\address{DII, Universit\`a degli Studi di Brescia, Via
Branze 38, I-25123 Brescia, Italy}
\email{alessandro.gnutti@unibs.it}

\author[S. Polak]{Sven Polak}
\address{Centrum Wiskunde \& Informatica (CWI), Amsterdam, the Netherlands}
\email{sven.polak@cwi.nl}

\begin{abstract}
A code $\mathcal{C} \subseteq \{0, 1, 2\}^n$ is said to be trifferent with length $n$ when for any three distinct elements of $\mathcal{C}$ there exists a coordinate in which they all differ. Defining $\mathcal{T}(n)$ as the maximum cardinality of trifferent codes with length $n$,
$\mathcal{T}(n)$ is unknown for $n \ge 5$. In this note, we use an optimized search algorithm to show that $\mathcal{T}(5) = 10$ and $\mathcal{T}(6) = 13$.
\end{abstract}


\maketitle

\section{Introduction}
Let $k \geq 3$ and $n \geq 1$ be integers, and let $\mathcal{C}$ be a subset of $\{0,1,\ldots,k-1\}^n$ with the property that for any $k$ distinct elements there exists a coordinate in which they all differ. A subset $\mathcal{C}$ with this property is called perfect $k$-hash code with length $n$ (perfect $3$-hash codes are called trifferent codes).  The problem of finding upper bounds for the maximum size of perfect $k$-hash codes is a fundamental problem in theoretical computer science. An elementary double counting argument, as shown in \cite{KornerMarton}, gives the following bound on the cardinality of $k$-hash codes:

\begin{equation}\label{eq:recursivebound}
    |\mathcal{C}| \leq (k-1) \cdot \left(\frac{k}{k-1}\right)^n \textit{ for every } k \geq 3\,.
\end{equation}

In 1984 Fredman and Koml\'os \cite{FredmanKomlos} improved the bound in \eqref{eq:recursivebound} for every $k \geq 4$ and sufficiently large $n$, obtaining the following result:

\begin{equation}\label{eq:fredmanKomlos}
    |\mathcal{C}| \leq \left(2^{k! / {k^{k-1}}}\right)^n.
\end{equation}

Additional refinements of this bound have been progressively achieved over the years. See for example \cite{Arikan2, Arikan1, DalaiJaikuGuru} for the case $k=4$,  \cite{CostaDalai} for the cases $k=5,6$, and \cite{KornerMarton, RiazanovGuru, DellaFioreCostaDalai1, DellaFioreCostaDalai2} for $k \geq 5$. For the sake of completeness, we mention that some improvements on the asymptotic probabilistic lower bounds on the maximum size of perfect $k$-hash codes have been recently obtained in \cite{XingYuan} for both small values of $k$ and $k$ sufficiently large.

In contrast, no recent progress has been made to improve the simple bound given in \eqref{eq:recursivebound} for $k = 3$. This bound has not been outperformed by any algebraic technique, including the recent slice-rank method by Tao \cite{TaoSliceRank}. Indeed, Costa and Dalai showed in \cite{CostaDalaiSliceRank} that the slice-rank method cannot be applied in a \textit{simple way} in order to improve the bound in \eqref{eq:recursivebound}. It is worth to mention that an improvement has been recently obtained in \cite{PohataZakharov}, however the authors restrict the codes to be linear, i.e., $\mathcal{C} \subset \mathbb{F}_3^n$ and $\mathcal{C}$ is a subspace of $\mathbb{F}_3^n$. 

As a consequence, particular attention is given to the case $k=3$.
Defining $\mathcal{T}(n)$ as the maximum cardinality of trifferent codes with length $n$, it is easy to verify that $\mathcal{T}(1)=3$, $\mathcal{T}(2)=4$ and $\mathcal{T}(3)=6$.
In addition, the authors in \cite{KornerMarton} showed that the so called \textit{tetra-code} is a trifferent code with length $4$ and cardinality $9$: this result leads to $\mathcal{T}(4)=9$.
To the best of our knowledge, $\mathcal{T}(n)$ is currently unknown for $n \ge 5$. In this note, we show that $\mathcal{T}(5) = 10$ and $\mathcal{T}(6) = 13$ and we use these results to refine the current best known upper bound on the cardinality of trifferent codes with length $n \geq 5$ (Section \ref{sec:results}).
The exact value is achieved by implementing an optimized algorithm in GAP which exhibits the non-existence of trifferent codes with lengths $5$ and $6$ and cardinalities $11$ and $14$, respectively (the algorithm description is given in Section \ref{sec:algorithm}).



\section{Improved upper bound on \texorpdfstring{$\mathcal{T}(n)$ for $n \geq 5$}{T(n) for n at least 5}}\label{sec:results}
The simple recursion used to obtain the bound in \eqref{eq:recursivebound} for $k=3$ is:
\begin{equation}\label{eq:recursion}
    \mathcal{T}(n) \leq \left\lfloor \frac{3}{2} \cdot \mathcal{T}(n-1) \right\rfloor,
\end{equation}
for every $n \geq 2$, with $\mathcal{T}(1)=3$.
Since $\mathcal{T}(4) = 9$, then $10 \leq \mathcal{T}(5) \leq \left\lfloor \frac{3}{2} \cdot 9 \right\rfloor = 13$.
The upper bound is obtained using \eqref{eq:recursion}, while the lower bound comes easily from the fact that
$
    \mathcal{T}(n) \geq \mathcal{T}(n-1)+1,
$
for every $n \geq 2$. Indeed, when a construction of a trifferent code with length $n-1$ is known, then it is always possible to trivially add an element of $\{0,1,2\}^n$ preserving the trifference property.
In Example \ref{ex:n5} we give a construction of a trifferent code with length $5$ and cardinality $10$ that is built using the tetra-code, see \cite{KornerMarton} for the definition. The $10$ elements of $\{0,1,2\}^5$ are represented in columns.
For $n=6$, we have that $13 \leq \mathcal{T}(6) \leq 19$. A trifferent code with length $6$ and cardinality $13$ is given in Example \ref{ex:n6}.


\begin{ex}[$\mathcal{T}(5) \geq 10$]\label{ex:n5}
\begin{center}
\begin{footnotesize}
\begin{equation*}
\begin{array}{|rrrrrrrrrr|}
 0 & 0 & 0 & 0 & 1 & 1 & 1 & 2 & 2 & 2 \\
 0 & 0 & 1 & 2 & 0 & 1 & 2 & 0 & 1 & 2 \\
 0 & 0 & 1 & 2 & 2 & 0 & 1 & 1 & 2 & 0 \\
 0 & 0 & 1 & 2 & 1 & 2 & 0 & 2 & 0 & 1\\
 0 & 1 & 2 & 2 & 2 & 2 & 2 & 2 & 2 & 2
\end{array}
\end{equation*}
\end{footnotesize}
\end{center}
\end{ex}


\begin{ex}[$\mathcal{T}(6) \geq 13$]\label{ex:n6}
\begin{center}
\begin{footnotesize}
\begin{equation*}
\begin{array}{|rrrrrrrrrrrrr|}
0&0&2&2&2&2&2&0&1&1&1&1&1 \\ 
0&1&0&2&2&1&2&2&0&1&1&2&1 \\
0&1&1&0&2&2&1&2&2&0&1&1&2 \\
0&1&1&1&0&2&2&2&2&2&0&1&1 \\
0&1&2&1&1&0&2&2&1&2&2&0&1 \\
0&1&1&2&1&1&0&2&2&1&2&2&0 \\
\end{array}
\end{equation*}
\end{footnotesize}
\end{center}
\end{ex}

We have designed an algorithm for searching trifferent codes with lengths $5$ and $6$ and cardinalities $11$ and $14$, respectively (see Section \ref{sec:algorithm} for the description of the algorithm). The search ended without returning any trifferent codes, thus proving that $\mathcal{T}(5) \leq 10$ and $\mathcal{T}(6) \leq 13$. Hence, the following theorem holds:

\begin{thm}\label{teo:nonexistence}
$
    \mathcal{T}(5) = 10 \text{ and } \mathcal{T}(6) = 13.
$
\end{thm}

This result allows us to focus on the current best known bounds on the maximum cardinality of trifferent codes, which can be expressed as $\mathcal{T}(n) \leq c \cdot (3/2)^n$, where $c$ is a constant and $n$ is sufficiently large. Since finding a better upper bound on the $\limsup_{n \to \infty} \sqrt[n]{\mathcal{T}(n)}$ is a very hard task, it becomes interesting to improve the constant $c$. The bound shown in \eqref{eq:recursivebound} gives us $c=2$, but a better constant can be obtained using \eqref{eq:recursion} and the fact that $\mathcal{T}(4) = 9$, that is $c = 9/ (3/2)^4 \approx 1.78$. We are able to improve this constant using Theorem \ref{teo:nonexistence} and \eqref{eq:recursion}. These statements directly imply:

\begin{cor}\label{teo:bestconstant}
\begin{align*}\mathcal{T}(n) \leq \frac{10}{(3/2)^5} \cdot \left(\frac{3}{2}\right)^n \approx 1.32 \cdot  \left(\frac{3}{2}\right)^n \text{ for every } n \geq 5, \\
\mathcal{T}(n) \leq \frac{13}{(3/2)^6} \cdot \left(\frac{3}{2}\right)^n \approx 1.15 \cdot  \left(\frac{3}{2}\right)^n \text{ for every } n \geq 6.
\end{align*}
\end{cor}

Since the floor function is involved in the recursive formula \eqref{eq:recursion}, we can improve the constant $c$ by iterating \eqref{eq:recursion} $m$ times starting from a fixed $n_0$ and a known upper bound on $\mathcal{T}(n_0)$. This results in the following theorem.

\begin{thm}\label{teo:improvedconst}
$\mathcal{T}(n) \leq 1.09 \cdot \left(\frac{3}{2}\right)^n$ for every $n \geq 12$.
\end{thm}
\begin{proof}
    Fix an integer $n_0 \geq 1$ and consider the following recursive formula that describes a sequence of achievable constants for $\mathcal{T}(n) \leq l(m) \cdot (3/2)^{n}$ when $n \geq n_0+m$:
\begin{equation}\label{eq:recursiveconstant}
   l(m) = \left\lfloor l(m-1) \cdot \left(\frac{3}{2}\right)^{n_0+m} \right\rfloor \cdot \left(\frac{3}{2}\right)^{-n_0-m} \textit{ for } m \geq 1,
\end{equation}
where $l(0) = \mathcal{T}(n_0) \cdot (3/2)^{-n_0}$. Taking $n_0 = 6$ and $m = 6$, we obtain the thesis.
\end{proof}
Since the sequence $l(m)$ is non-increasing, we are interested in the $\lim_{m \to \infty} l(m)$.
 Computing that limit is not trivial, so we use the following recursive relation to obtain a lower bound:
\begin{equation*}
   d(m) = d(m-1) - \frac{1}{2} \cdot \left(\frac{3}{2}\right)^{-n_0-m} \textit{ for } m \geq 1,
\end{equation*}
where $d(0) = \mathcal{T}(n_0) \cdot (3/2)^{-n_0}$.
It is easy to see that $l(m) \geq d(m)$ for every $m \geq 0$. Then we have:
\begin{align}
\nonumber
  \lim_{m \to \infty} d(m) &= \lim_{m \to \infty} \left(\mathcal{T}(n_0) - \frac{1}{2} \cdot \sum_{i=1}^m \left(\frac{2}{3}\right)^i\right) \cdot \left(\frac{3}{2}\right)^{-n_0} \\
  \label{eq:constantlowerbound}
  &=\left(\mathcal{T}(n_0) - 1\right) \cdot \left(\frac{3}{2}\right)^{-n_0}.
\end{align}
\begin{rem}
Since $\mathcal{T}(4) = 9$, if we fix $n_0 = 4$ then we can substitute them into \eqref{eq:constantlowerbound} to get that $\lim_{m \to \infty} l(m) \geq \lim_{m \to \infty} d(m) = 8 \cdot (3/2)^{-4} \approx 1.59$. This lower bound is, in any case, greater than the constant that we have found in Theorem \ref{teo:improvedconst}.
\end{rem}

\section{Proof of Theorem \ref{teo:nonexistence} - The Algorithm}\label{sec:algorithm}
Computing a brute-force search for finding a trifferent code with length $n$ and cardinality $M$ would require to test $\binom{3^n}{M}$  subsets, and for each of them compare $\binom{M}{3}$ triplets: overall, for $(n,M)=(5,11)$ one would test $\approx 10^{20}$ triplets while for $(n,M)=(6,14)$ one would test approximately  $10^{30}$ triplets. These numbers are prohibitively large.

Our algorithm dramatically reduces the number of operations, without missing any potential trifferent code.
First, we list the elements of $\mathcal{C}_n = \{0, 1, 2\}^n$ in lexicographic order and fix $(i_1, i_2, \ldots, i_{M})$ as the indices representing the $M$ elements to test, requiring that $i_1<i_2<\ldots <i_{M}$.
Then, let $\mathcal{C}^m_n$ be the code containing the elements associated to the first $m$ indices.
Starting from $m=3$, we check if $\mathcal{C}^m_n$ is trifferent: based on the output, the variable $m$ and the indices are updated accordingly to the pseudocode reported in Algorithm 1.

\begin{algorithm}
\label{alg:pseudocode}
\caption{Check if $\mathcal{T}(n) \geq M$.}\label{alg:cap}
\begin{algorithmic}
\Require $\left(c(1), \ldots, c(3^n)\right) = \{0,1,2\}^n$ ordered lexicographically,
\State $\left(i_1, \ldots, i_{M}\right) \gets (1, \ldots, M)$, $m \gets 3$
\Repeat
\If{$\left\{c(i_1),\ldots,c(i_{m})\right\}$ \textbf{is} trifferent \textbf{or} $m < 3$} 
    \IIf{$m=M$} \Return True \EndIIf \State $m \gets m+1$
\Else 
\State $m' \gets $ \textbf{min}$\left\{m'' : i_{m''} \geq  3^n-M+{m''}\right\}$
\If{$m'$ \textbf{exists}}
    \State $m \gets m'-1$, $i_m \gets i_m + 1$
    \State $i_{t+1} \gets i_{t} + 1, \textit{ for every } m \leq t \leq M-1$
\Else
    \State $i_{t} \gets i_{t} + 1, \textit{ for every } m \leq t \leq M$
\EndIf
\EndIf
\Until{$i_1 \geq 2$}
\State \Return False
\end{algorithmic}
\end{algorithm}
\noindent
At each update, $\mathcal{C}^m_n$ is tested: however, only the triplets containing the $i_m$-th element have to be examined, since all the other triplets have been already verified by construction. This is the first key point of our algorithm.


In addition, we are able to force some restrictions on the set of the indices. Two codes $C,D \subseteq \mathbb{F}_3^n$ are called \emph{equivalent} if $D$ be obtained from $C$ by subsequently applying permutations to the coordinate positions and to the symbols $\{0,1,2\}$ in each coordinate.
Given a trifferent code, by symmetry we can find an equivalent code containing the zero vector and a vector of the form $(0, \ldots, 0, 1, \ldots, 1)$, and not containing  nonzero words lexicographically smaller than this vector. As a consequence, our algorithm stops the search of trifferent codes immediately when $i_1 = 2$, and limits the set of values that the second index can assume, namely, $i_2 \in \{\frac{3^{i}+1}{2} : i = 1,\ldots, n\}$.  
Furthermore, suppose there exists a trifferent code $\mathcal{C}$ with length $n$ and cardinality $M$. Let $s_0, s_1, s_2$ be the number of elements in $\mathcal{C}$ with symbols $0$, $1$ and $2$, respectively, at the first coordinate. It is easy to see that $s_i+s_j \leq \mathcal{T}(n-1)$ for $i \neq j$, so $s_0, s_1, s_2 \geq M - \mathcal{T}(n-1) $. It means that we should have for each symbol $0$, $1$ and $2$ at least $M - \mathcal{T}(n-1)$ elements in $\mathcal{C}$ with that symbol in the first coordinate. As a consequence, recalling that we list the elements of $\mathcal{C}_n$ in lexicographic order, we can force:
\begin{itemize}
    \item $i_{M - \mathcal{T}(n-1)}\leq 3^{n-1}$ (first coordinate equal to 0);
    \item $i_{2(M - \mathcal{T}(n-1))}\leq 2\cdot 3^{n-1}$ (first coordinate equal to 1);
    \item $i_{2\mathcal{T}(n-1)-M+1}> 3^{n-1}$ (first coordinate equal to 1);
    \item $i_{\mathcal{T}(n-1)+1}> 2 \cdot 3^{n-1}$ (first coordinate equal to 2).
\end{itemize}
For the sake of readability, the pseudocode reported in Algorithm 1 does not include the restrictions on the set of the indices. However, the code associated to the final version of the algorithm is publicly available and can be found at \cite{DellaFioreSito}.

We have executed our program for $(n,M) = (5, 11)$ and $(n,M)=(6, 14)$, and no trifferent code has been found. The total number of tested triplets is $\approx 10^{7}$ for $(n,M) = (5, 11)$ and $\approx 10^{11}$ for $(n, M) = (6, 14)$, thus saving a factor of $\approx 10^{13}$ and $\approx 10^{19}$, respectively, compared to the full brute-force strategy.

As a side note: inspired by a semidefinite programming upper bound for cap sets~\cite{Gijswijt}, we could alternatively obtain the upper bound~$\mathcal{T}(5)\leq 10$ using the method from~\cite{LitjensPolakSchrijver}, in which all extra constraints from Eq.\ (3) of~\cite{LitjensPolakSchrijver}  were included to obtain the bound.

\begin{rem}
For~$(n,M)=(6,13)$, the search returned a set $S$ of~$1046$ trifferent codes up to symmetry choices explained above. 
For any code in $S$, we generated all equivalent codes and deleted the ones contained in $S$ from $S$. 
We had to repeat this 3 times until the set was empty. So there are 3 distinct trifferent (n,M)=(6,13)-codes up to equivalence. These are:
\begin{footnotesize}
\begin{equation*}
\arraycolsep=1.2pt\def\arraystretch{.8}
\begin{array}{|rrrrrrrrrrrrr|}
0&0&2&2&2&2&2&0&1&1&1&1&1 \\
0&1&0&2&2&2&1&2&0&1&1&1&2 \\
0&1&1&0&2&1&2&2&2&0&1&2&1 \\
0&1&1&1&0&2&2&2&2&2&0&1&1 \\
0&1&1&2&1&0&2&2&2&1&2&0&1 \\
0&1&2&1&1&1&0&2&1&2&2&2&0 \\
\end{array}\,,\quad
\begin{array}{|rrrrrrrrrrrrr|}
0&0&2&2&2&2&2&0&1&1&1&1&1 \\ 
0&1&0&2&2&1&2&2&0&1&1&2&1 \\
0&1&1&0&2&2&1&2&2&0&1&1&2 \\
0&1&1&1&0&2&2&2&2&2&0&1&1 \\
0&1&2&1&1&0&2&2&1&2&2&0&1 \\
0&1&1&2&1&1&0&2&2&1&2&2&0 \\
\end{array}\,,\quad
\begin{array}{|rrrrrrrrrrrrr|}
0&0&2&2&2&2&2&0&1&1&1&1&1 \\
0&1&0&2&2&1&2&2&0&1&1&2&1 \\ 
0&1&1&0&2&2&2&2&2&0&1&1&1 \\
0&1&1&1&0&2&1&2&2&2&0&1&2 \\
0&1&2&1&1&0&2&2&1&2&2&0&1 \\
0&1&1&1&2&1&0&2&2&2&1&2&0 \\
\end{array}\,.
\end{equation*}\end{footnotesize}For each of these codes, in each coordinate position two symbols occur 5 times and one symbol occurs 3 times.
\end{rem}

\end{document}